\documentclass[10pt,leqno]{amsart}
\usepackage{amsfonts,amssymb}

\usepackage{amsmath}
\usepackage{amsthm}
\usepackage{amsrefs}
\usepackage{qsymbols}
\usepackage{latexsym}
\usepackage{chngcntr}
\usepackage[noadjust]{cite}
\usepackage{paralist}
\usepackage{doi}
\usepackage{esint}
\newtheorem{theorem}{Theorem}[section]

\newtheorem{lemma}[theorem]{Lemma}
\newtheorem{proposition}[theorem]{Proposition}
\newtheorem{corollary}[theorem]{Corollary}
\theoremstyle{definition}

\newtheorem{remarks}[theorem]{Remarks}

\newcommand{\IR}{\mathbb{R}}

\newcommand{\IN}{\mathbb{N}}

\newcommand{\IE}{\mathbb{E}}
\newcommand{\IP}{\mathbb{P}}

\newcommand{\II}{\mathbb{I}}
\newcommand{\cO}{\mathcal{O}}

\newcommand{\cF}{\mathcal{F}}
\newcommand{\cS}{\mathcal{S}}

\newcommand{\cL}{\mathcal{L}}

\newcommand{\vn}{\vspace{.1cm}\noindent}
\newcommand{\divergence}{\operatorname{div}}

\numberwithin{equation}{section} 

\title[Rigorous justification of the hydrostatic approximation]{Rigorous justification of the hydrostatic approximation for the primitive equations by scaled Navier-Stokes equations}

\author[Furukawa]{Ken Furukawa}
\address{Graduate School of Mathematical Sciences, University of Tokyo, Komaba 3-8-1, Meguro-ku, Tokyo, 153-8914, Japan }
\email{kenf@ms.u-tokyo.ac.jp}

\author[Giga]{Yoshikazu Giga} 
\address{Graduate School of Mathematical Sciences, University of Tokyo, Komaba 3-8-1, Meguro-ku, Tokyo, 153-8914, Japan }
\email{labgiga@ms.u-tokyo.ac.jp}

\author[Hieber]{Matthias Hieber} 
\address{Department of Mathematics,
	TU Darmstadt, Schlossgartenstr. 7, 64289 Darmstadt, Germany}
\email{hieber@mathematik.tu-darmstadt.de}

\author[Hussein]{Amru Hussein} 
\address{Departement of Mathematics,
	TU Darmstadt, Schlossgartenstr. 7, 64289 Darmstadt, Germany}
\email{hussein@mathematik.tu-darmstadt.de}

\author[Kashiwabara]{Takahito Kashiwabara}
\address{Graduate School of Mathematical Sciences, The University of Tokyo, 3-8-1 Komaba, Meguro, Tokyo 153-8914, Japan}
\email{tkashiwa@ms.u-tokyo.ac.jp}

\author[Wrona]{Marc Wrona}
\address{Departement of Mathematics,
	TU Darmstadt, Schlossgartenstr. 7, 64289 Darmstadt, Germany}
\email{wrona@mathematik.tu-darmstadt.de}

\keywords{primitive equations, scaled Navier-Stokes equations, strong convergence, convergence rate \\
	This work was partly supported by the DFG International Research Training Group IRTG 1529 and the JSPS Japanese-German Graduate Externship on Mathematical Fluid Dynamics. 
	The second author is partly supported by JSPS through grant Kiban S (No. 26220702), Kiban A (No. 17H01091), Kiban B (No. 16H03948) and the fourth and the last author are supported 
by IRTG 1529 at TU Darmstadt. The fifth author was supported by JSPS Grant-in-Aid for Young Scientists B (No.
17K14230)}

\begin{document}

\begin{abstract}
Consider the anisotropic Navier-Stokes equations as well as the primitive equations. It is shown that the horizontal velocity of the solution to the anisotropic 
Navier-Stokes equations in a cylindrical domain of height $\varepsilon $ with initial data $u_0=(v_0,w_0)\in B^{2-2/p}_{q,p}$,  $1/q+1/p\le 1$ if $q\ge 2$ and $4/3q+2/3p\le 1$ if $q\le 2$, converges  as $\varepsilon \to 0$ with convergence rate $\cO (\varepsilon )$ to the horizontal velocity 
of the solution to the primitive equations with initial data $v_0$ with respect to the maximal-$L^p$-$L^q$-regularity norm. Since the difference of the corresponding vertical velocities 
remains bounded with respect to that norm, the convergence result yields  a  rigorous justification of the hydrostatic approximation in the primitive equations in this setting.  It generalizes in particular  
a result by Li and Titi for the $L^2$-$L^2$-setting. The approach presented here does not rely on second order energy estimates but on maximal $L^p$-$L^q$-estimates for the heat equation.
\end{abstract}

\maketitle

\section{Introduction}
The primitive equations for the ocean and atmosphere are considered to be a fundamental model for geophysical flows, see e.g. the survey article \cite{Li2016}. The mathematical analysis of 
these equations has been pioneered by Lions, Teman and Wang  in their articles \cite{Lions1, LionsTemanWang1992, Lions:1993}, where they proved the existence of global, weak solutions
to the primitive equations. Their uniqueness remains an open problem until today. Global strong well-posedness of the primitive equations for initial data in $H^1$ was shown by Cao and Titi 
in \cite{CaoTiti2007} using energy methods. A different approach, based on the theory of  evolution equations, was introduced by Hieber and Kashiwabara in \cite{Hieber2016} and 
subsequent works \cite{HieberKashiwabaraHussein2016, GigaGriesHusseinHieberKashiwabara2016, NeumannNeumann, DirichletNeumann}. 

It is the aim of this paper to show that the primitive equations can be obtained as the limit of anisotropically scaled  Navier-Stokes equations. The  scaling parameter $\varepsilon>0$ represents
the ratio of the depth to the horizontal width. Such an approximation is motivated by the fact that for large-scale oceanic dynamics, this aspect ratio $\varepsilon$ is rather small and 
implies anisotropic viscosity coefficients 
%which are   horizontally %of order $\mathcal{O}(1)$ and vertically of order $\mathcal{O}(\varepsilon^2)$ 
(see e.g. \cite{Pedlosky1979}). For an aspect ratio $\varepsilon $, i.e., in the case where the spacial domain can be represented as $\Omega _\varepsilon =G\times (-\varepsilon , +\varepsilon) $ for 
some $G\subset \IR ^2$, and a horizontal and vertical eddy viscosity $1$ and $\varepsilon ^2$, respectively, the system can be rescaled into the form
\begin{equation}
\left \{\begin{array}{rll}
\partial _tv_\varepsilon+u_\varepsilon\cdot\nabla v_\varepsilon-\Delta v_\varepsilon+\nabla _Hp_\varepsilon= & 0,\\
\varepsilon (\partial _tw_\varepsilon+u_\varepsilon\cdot \nabla w_\varepsilon-\Delta w_\varepsilon )+\tfrac{1}{\varepsilon }\partial _zp_\varepsilon= & 0,\\
\divergence u_\varepsilon =&  0.
\end{array}\right .
\end{equation}
in the time-space domain $(0,T)\times \Omega _1$, which  is \emph{independent} of the aspect ratio. We refer  to \cite{LiTiti2017} for more details on this rescaling procedure.
Here the horizontal and vertical velocities $v_\varepsilon $ and $w_\varepsilon$ describe the three-dimensional velocity $u_\varepsilon$, while $p_\varepsilon$ denotes the pressure of the fluid. 
Here  $\partial _z$ denotes the vertical-derivative, $\nabla _H$ and $\divergence _H$ the horizontal gradient and divergence, whereas $\divergence $, $\nabla$
and $\Delta$ stand for the usual three-dimensional spatial divergence, gradient, and Laplacian.

First convergence results for the above system in the steady state case go back to Besson and Laidy  \cite{BessonLaydi1992}. The convergence of the above system has been studied first 
by  Az\' erad and Guill\' en in \cite{ArezadGuillen2001} in the setting of  weak convergence, no uniform convergence rate was given. 

Recently, Li and Titi \cite{LiTiti2017} investigated the strong convergence of the above system within  the $L^2$-$L^2$-setting for horizontal initial velocities belonging to $H^1$ and $H^2$. 
In addition, they showed a convergence rate of order $\mathcal{O}(\varepsilon)$. 

It is the aim of this paper to show convergence results of the above system  in the strong sense within the $L^p$-$L^q$-setting.  Our method is very different from the one introduced by 
\cite{LiTiti2017}, whereas they rely on second order energy estimates, our approach is based on maximal $L^p$-$L^q$-regularity estimates for the heat equation. This allows us to give a very short proof of the 
convergence result in the more general $L^p$-$L^q$-setting, which even in the $L^2$-$L^2$-setting allows for a   slightly larger class of initial data compared to the one introduced   
by Li and Titi in \cite{LiTiti2017} by using energy estimates.

\section{Preliminaries}
Consider the cylindrical domain $\Omega :=(0,1)^2\times (-1,1)$. Let $u=(v,w)$ be the solution of the primitive equations
\begin{equation}\label{eq_PE}
\left \{\begin{array}{rll}
\partial _tv+u\cdot\nabla v-\Delta v+\nabla _Hp=& \ 0&\text{ in }(0,T)\times\Omega ,\\
\partial _zp=&\ 0&\text{ in }(0,T)\times\Omega ,\\
\divergence u=& \ 0&\text{ in }(0,T)\times\Omega ,\\
p\text{ periodic in }x,y& %\text{ odd }&\text{ in }z,
\\
v ,w \text{ periodic in }x,y,z,&\text{ even and odd }&\text{ in }z,\\
u(0)=&\ u_0&\text{ in }\Omega ,
\end{array}\right .\tag{PE}
\end{equation}
and $u_\varepsilon =(v_\varepsilon ,w_\varepsilon )$ be the solution of the anisotropic Navier-Stokes equations
\begin{equation}\label{NSe}
\left \{\begin{array}{rll}
\partial _tv_\varepsilon+u_\varepsilon\cdot\nabla v_\varepsilon-\Delta v_\varepsilon+\nabla _Hp_\varepsilon=& \ 0&\text{ in }(0,T)\times\Omega ,\\
\partial _tw_\varepsilon+u_\varepsilon\cdot \nabla w_\varepsilon-\Delta w_\varepsilon+\tfrac{1}{\varepsilon ^2}\partial _zp_\varepsilon=&\ 0&\text{ in }(0,T)\times\Omega ,\\
\divergence u_\varepsilon =& \ 0&\text{ in }(0,T)\times\Omega ,\\
p_\varepsilon\text{ periodic in }x,y,z,&\text{ even }&\text{ in }z,\\
v _\varepsilon,w _\varepsilon\text{ periodic in }x,y,z,&\text{ even and odd }&\text{ in }z,\\
u_\varepsilon (0)=&\ u_0&\text{ in }\Omega .
\end{array}\right .\tag{NS$_\varepsilon$}
\end{equation}
Here $v$ and $v_\varepsilon $ denote the (two-dimensional) horizontal velocities, $w$ and $w_\varepsilon $  the vertical velocities, and $p$ and $p_\varepsilon $ denote the pressure term for 
the primitive equations as well as the Navier-Stokes equations, respectively. These are functions of three space variables $x,y\in (0,1)$, $z\in (-1,1)$. The vertical periodicity and parity conditions 
correspond to an equivalent set of equations with vertical Neumann boundary conditions for the horizontal velocity and vertical Dirichlet boundary conditions for the vertical velocity (cf. e.g. \cite{CaoLiTiti2015}). Since $w$ is odd, the divergence free condition 
for the primitive equation translates into 
$\divergence _H\overline{v} =0$, where $\overline{v}(x,y)=\frac{1}{2}\int^1_{-1}v(x,y,z)\mathop{}\!\mathrm{d}z$, and 
\[w(\cdot,\cdot,z)= -\int^z_{-1}\divergence _H v(\cdot,\cdot,\zeta)\mathop{}\!\mathrm{d}\zeta.
\]

For $p,q\in (1,\infty )$ and $s\in [0,\infty )$ we define the Bessel potential and Besov spaces
\[H^{s,p}_{per}(\Omega )=\overline{C_{per}^\infty (\overline{\Omega })}^{\Vert \cdot\Vert _{H^{s,p}}} \quad \hbox{and} \quad B^{s}_{p,q,per}(\Omega )=\overline{C_{per}^\infty (\overline{\Omega })}^{\Vert \cdot\Vert _{B^{s}_{p,q}}},
\]
where
$C_{per}^\infty (\overline{\Omega })$ denotes the space of smooth functions that are periodic of any order (cf. \cite[Section 2]{Hieber2016}) in all three directions 
on $\partial \Omega$. The space $H^{s,p}(\Omega )$ denotes the Bessel potential space of order $s$, with norm $\Vert \cdot \Vert _{H^{s,p}}$ defined via the restriction of the corresponding space defined on the whole space to $\Omega$ (cf. \cite[Definition 3.2.2.]{Triebel}). Moreover, $B^{s}_{p,q}(\Omega )$ denotes a Besov space on $\Omega$, which is defined by restrictions of functions on the whole space 
to $\Omega$, see e.g. \cite[Definition 3.2.2.]{Triebel}. Note that $L^p(\Omega )=H^{0,p}_{per}(\Omega )$ and $B^{s}_{p,2,per}(\Omega ) = H^{s}_{p,per}(\Omega )$. The anisotropic structure of the primitive equations motivates the definition of the Bessel potential spaces $H^{s,p}_{xy}:=H^{s,p}((0,1)^2)$ and $H^{s,p}_{z}:=H^{s,p}(-1,1)$ for the horizontal and vertical variables, respectively. Similarly  as above we write 
$L^p_{xy} :=H^{0,p}_{xy}$ and $L^p_z:=H^{0,p}_{z}$ and set $H^{s,p}_{xy}H^{r,q}_{z}:=H^{s,p}((0,1)^2;H^{r,q}_{z})$.

The divergence free conditions in the above sets of equations can be encoded into the space of solenoidal functions
\[
L^p_\sigma (\Omega )=\overline{\{u\in C_{per}^\infty (\overline{\Omega })^3:\divergence u =0\}}^{\Vert \cdot\Vert _{L^p}}\quad \mbox{and} \quad 
L^p_{\overline{\sigma }} (\Omega )=\overline{\{v\in C_{per}^\infty (\overline{\Omega })^2:\divergence _H\overline{v} =0\}}^{\Vert \cdot\Vert _{L^p}}.
\]

For given $p,q\in (1,\infty)$ we set
\begin{align*}
X_0 &:=L^q(\Omega ), \quad X_1:=H^{2,q}_{per}(\Omega ),\\
X^v_0&:=\{ v\in L^q_{\overline{\sigma}}(\Omega ):v\text{ even in }z\}, \quad X^v_1:=\{ v\in H^{2,q}_{per}(\Omega )^2 \cap L^q_{\overline{\sigma}}(\Omega ) :v\text{ even in }z\},\\
X_0^u&:=\{(v_1,v_2,w)\in L^q_\sigma (\Omega ):v_1,v_2\text{ even }w\text{ odd in }z\}, \\
X_1^u&:=\{(v_1,v_2,w)\in H^{2,q}_{per}(\Omega)^3\cap L^q_\sigma (\Omega ):v_1,v_2\text{ even }w\text{ odd in }z\},
\end{align*}
and consider the trace space $X_\gamma$ defined by
\begin{align*}
X_\gamma &=(X_0^u,X_1^u)_{1-1/p,p}.
\end{align*}
Here  $(\cdot,\cdot)_{1-1/p,p}$ denotes the real interpolation functor.

Following the lines of  \cite[Section 4]{HieberKashiwabaraHussein2016} and \cite{GigaGriesHieberHusseinKashiwabara2017} the trace space $X_\gamma$ 
can be characterized as follows.

\vn
\begin{lemma}[Characterization of the trace space]
Let $p,q\in (1,\infty )$. Then 
\[
X_\gamma =\left \{ \begin{array}{lll}
\{(v_1,v_2,w)\in B^{2-2/p}_{q,p,per}(\Omega )^3\cap L^q_\sigma (\Omega ):&v=(v_1,v_2)\text{ even, }w\text{ odd in }z,&\\
& (\partial _zv, w)=0\text{ at }z=-1,0,1\} ,&1>\frac{2}{p}+\frac{1}{q},\\
\{(v_1,v_2,w)\in B^{2-2/p}_{q,p,per}(\Omega )^3\cap L^q_\sigma (\Omega ):&v=(v_1,v_2)\text{ even, }w\text{ odd in }z,&\\
& w=0\text{ at }z=-1,0,1\} ,& 1<\frac{2}{p}+\frac{1}{q}.
\end{array} \right .
\]
\end{lemma}

For $p,q\in (1,\infty)$ and $T\in (0,\infty ]$ we also define the maximal regularity spaces 
\[
\IE _0(T):=L^p(0,T;X_0),\quad \IE _1(T):=L^p(0,T;X_1)\cap H^{1,p}(0,T;X_0),
\]
and analogously $\IE _0^v(T)$, $\IE _1^v(T)$ and $\IE _0^u(T)$, $\IE _1^u(T)$ with respect to $X_0^v$, $X_1^v$ and $X_0^u$, $X_1^u$, respectively. In order to simplify our notation we sometimes omit the 
subscripts $u$ and $v$ and write only $\IE _0(T)$ and $\IE _1(T)$. 

Finally, we say that $u=(v,w)$ is a \emph{strong solution to the primitive equations} (in the $L^p$-$L^q$-setting), if $v\in \IE _1^v$ and  (\ref{eq_PE}) holds almost everywhere. We say that 
$u_\varepsilon$ is a \emph{strong solution to the Navier-Stokes equations}, if $u\in \IE _1^u$ and  (\ref{NSe}) holds almost everywhere.

\section{Main Result}

Roughly speaking, the idea of our approach consists of  controling the maximal regularity norm of the differences $(v_\varepsilon -v,\varepsilon (w_\varepsilon -w))$ by 
the aspect ratio $\varepsilon$. To this end, we introduce the difference equations of (\ref{NSe}) and (\ref{eq_PE}): setting  $V_\varepsilon:=v_\varepsilon-v$, $W_\varepsilon:=w_\varepsilon-w$, 
$U_\varepsilon:=(V_\varepsilon,W_\varepsilon)$ and $P_\varepsilon =p_\varepsilon -p$, we obtain 
\begin{equation}\label{eq_diff-eq}
\left \{\begin{array}{rll}
\partial _tV_\varepsilon-\Delta V_\varepsilon=& \ F_H(V_\varepsilon ,W_\varepsilon )-\nabla _HP_\varepsilon&\text{ in }(0,T)\times\Omega ,\\
\partial _t\varepsilon W_\varepsilon -\Delta \varepsilon W_\varepsilon=&\ \varepsilon F_z(V_\varepsilon ,W_\varepsilon )-\tfrac{1}{\varepsilon }\partial _zP_\varepsilon&\text{ in }(0,T)\times\Omega ,\\
\divergence U_\varepsilon =& \ 0&\text{ in }(0,T)\times\Omega ,\\
P_\varepsilon\text{ periodic in }x,y,z,& \text{ even }&\text{ in }z,  \\
V_\varepsilon ,W_\varepsilon \text{ periodic in }x,y,z,&\text{ even and odd }&\text{ in }z,\\
U_\varepsilon (0)=&\ 0&\text{ in }\Omega ,
\end{array}\right .
\end{equation}
where the forcing terms $F_H$ and $F_z$ are given by
\begin{align*}
F_H(V_\varepsilon,W_\varepsilon):=&\ -U_\varepsilon \cdot \nabla v-u\cdot \nabla V_\varepsilon -U_\varepsilon \cdot\nabla V_\varepsilon ,&\ \\
F_z(V_\varepsilon ,W_\varepsilon ):=&\ -U_\varepsilon \cdot \nabla w-u\cdot \nabla W_\varepsilon -U_\varepsilon \cdot \nabla W_\varepsilon -\partial _tw-u\cdot\nabla w+\Delta w.
\end{align*}
Applying  the maximal regularity estimate given in Proposition~\ref{prop_mr_estimate} to \eqref{eq_diff-eq}, we are able to estimate $\|U_\varepsilon \|_{\IE_1}$ in terms of the right hand sides. 
The latter  will be estimated in a series of lemmas in Section~\ref{se_nonlinear}. Like this we obtain a quadratic inequality for the norm of the differences 
(cf. Corollary~\ref{coro_q_ineq}) and we need to ensure that the constant term as well as the coefficient in front of the linear term are sufficiently small. This can be achieved provided 
the aspect ratio $\varepsilon $ is small enough and provided  the vertical and horizontal solution of the primitive equations exist globally in the maximal regularity class (cf. Proposition~\ref{prop_w_in_E1}).

\vspace{.1cm}\noindent
{\bf Assumption (A)}. Let $q\in \left (\frac{4}{3},\infty \right )$ and $p\ge \max \left \{ \frac{q}{q-1};\frac{2q}{3q-4}\right \}$, i.e., 
$1\ge\left \{\begin{array}{rl}
\frac{1}{p}+\frac{1}{q},&\!\text{if }q\ge 2,\\
\frac{2}{3p}+\frac{4}{3q},&\!\text{if }q\le 2.
\end{array}
\right .
$

\noindent
We are now in the position to state our main result.

\vn
\begin{theorem}[Main Theorem]\label{mthm_MR}
Assume that  $p,q $ fulfill Assumption (A) and let $u_0\in X_\gamma $, $T>0$ and $(v,w)$ and  $(v_\varepsilon ,w_\varepsilon )$ be solutions of (\ref{eq_PE}) and (\ref{NSe}), respectively. 
Then there exists a constant $C>0$, independent of $\varepsilon$, such that  
for $\varepsilon $ sufficiently small it holds$$
\Vert (V_\varepsilon ,\varepsilon W_\varepsilon)\Vert _{\IE _1(T)}\le C\varepsilon. 
$$
In particular
\[(v_\varepsilon ,\varepsilon w_\varepsilon )\to (v,0)\text{ in }L^p(0,T;H^{2,q}(\Omega ))\cap H^{1,p}(0,T;L^q(\Omega ))
\]
as $\varepsilon \to 0$ with convergence rate $\cO (\varepsilon )$. 
\end{theorem}

\vn
\begin{remarks} 
a) If the solution  $u=(v,w)$ of the primitive equations exists globally in time, the convergence rate is uniform for all $T \in (0,\infty]$, see  
Remark~\ref{remunif} b). For example, if $p=q=2$ and the initial data are  mean value free, one can show that the solution to the primitive equations exists globally in $\IE _1(T)$ with $T=\infty$. \\
b) We note that the case $p=q=2$, investigated  before in \cite{LiTiti2017}, is covered by our result. More specifically, they assumed $v_0\in H^2$ whereas for our purposes $v_0,\divergence_H v_0\in H^1$ suffices. \\
c) The scaled Navier-Stokes equations are locally well-posed in the maximal regularity spaces on the torus and the parity conditions are preserved. Our main result, Theorem~\ref{mthm_MR}, yields that
for each time $T$ there exists an $\varepsilon>0$ such that the solution exists on $(0,T)$.\\ 
d) The primitive equations are well-posed for all times $T>0$ provided  $u_0 \in X_\gamma$, see  \cite{GigaGriesHieberHusseinKashiwabara2017}.	\\
e) Our method can be adjusted to the case with perturbed initial data. That is, given initial data $(u_{0,\varepsilon})_{\varepsilon>0}\subset X_\gamma $ converging to $u_0$ in $X_\gamma $ as $\varepsilon \to 0$ of order $\cO (\delta _\varepsilon )$ for some null-sequence $(\delta _\varepsilon )_{\varepsilon>0}$, then Theorem~\ref{mthm_MR} holds with $(v_\varepsilon ,w_\varepsilon )$ replaced by the solution of (\ref{NSe}) with initial data $v_{0,\varepsilon}$. In that case the maximal regularity norm of the differences is bounded by $C\max\{\varepsilon,\delta_\varepsilon \}$ and consequently the convergence rate is of order $\cO(\max\{\varepsilon,\delta_\varepsilon \})$.\\
\end{remarks}

\section{Proof of the Main Theorem}\label{se_nonlinear}

The proof of Theorem~\ref{mthm_MR} relies upon estimates on the terms $F_H$ and $F_z$ in equations (\ref{eq_diff-eq}) within the $L^p$-$L^q$-framework. These estimates imply eventually a quadratic 
inequality for the difference of the velocities, see Corollary~\ref{coro_q_ineq}. In order to establish these estimates we need to ensure that the solution of the primitive equations belongs to the 
maximal regularity class (see Proposition~\ref{prop_w_in_E1}) and  that the nonlinear terms can be estimated in $\IE_0(T)$, see Lemma~\ref{le_rhs_e0_easy} and \ref{le_rhs_e0_diffi}.
We hence subdivide our proof in three steps. Throughout this section let $T<\infty$.

\subsection{Nonlinear estimates}
In this subsection, we estimate the  bilinear terms and keep track of the $T$-dependence of the norms involved.    

\vn
\begin{lemma}\label{le_emb}
Let $s\ge s^\prime \ge 0$ and $p\in [1,\infty )$. Then $H^{s,p}(0,T)\overset{s-s^\prime }{\hookrightarrow }H^{s^\prime ,p}(0,T)$ where $\overset{\eta }{\hookrightarrow }$ stands for an embedding with 
embedding constant $CT^\eta$, $C>0$ independent of $T$.
\end{lemma}

\begin{proof}
Set $m:=\lfloor p(s-s^\prime )\rfloor$ and $1/r:=m+1/p+s^\prime -s\in [1/p,1)$. Sobolev embeddings and H\"older's inequality yield 
\[ 
H^{s,p}\hookrightarrow H^{m+s^\prime ,r}\overset{1-\frac{1}{r}}{\hookrightarrow } H^{m+s^\prime ,1}\hookrightarrow H^{m-1+s^\prime ,\infty}\overset{1}{\hookrightarrow} H^{m-1+s^\prime ,1}\hookrightarrow
\ldots \hookrightarrow H^{s^\prime ,\infty}\overset{\frac{1}{p}}{\hookrightarrow} H^{s^\prime ,p}.\qedhere
\]
\end{proof}

We will make use of the following classical  Mixed Derivative Theorem, see e.g. \cite[Corollary 4.5.10]{PruessSimonett}.

\vn
\begin{proposition}[Mixed Derivative Theorem]\label{prop_mdt}
If $\theta \in [0,1]$, then 
\[ 
\IE _1(T)\hookrightarrow H^{\theta ,p}(0,T;H^{2-2\theta ,q}(\Omega )).
\]
\end{proposition}

\vn
\begin{lemma}\label{le_rhs_e0_easy}
Let $p,q\in(1,\infty )$ such that $2/3p+1/q\le 1 $. Then  for all $v_1,v_2\in \IE _1(T)$, $\partial \in \{\partial _x,\partial _y,\partial _z\}$ and 
$\eta \in \left [0,\tfrac{3}{2}\left (1-\tfrac{2}{3p}-\tfrac{1}{q}\right )\right ]$ there exists a constant $C>0$ such that 
\[ 
\Vert v_1 \partial v_2\Vert _{\IE _0(T)}\le CT^{\eta }\Vert v_1\Vert _{\IE _1(T)}\Vert v_2\Vert _{\IE _1(T)}, \quad C>0.
\] 
\end{lemma}

\begin{proof}
Set $\theta _1=\tfrac{2\eta }{3} +\tfrac{2}{3p}$ and $\theta _2=\tfrac{1}{2}\theta _1$. The Mixed Derivative Theorem, Lemma~\ref{le_emb} and Sobolev's embedding yield 
\begin{align*}
&\IE _1(T)\hookrightarrow H^{\theta _1,p}(0,T;H^{2-2\theta _1,q})\overset{2\eta /3}{\hookrightarrow }H^{2/3p,p}(0,T;H^{2-2\theta _1,q})\hookrightarrow L^{3p}(0,T;L^{3q}(\Omega )),\\
& \IE _1(T)\hookrightarrow H^{\theta _2,p}(0,T;H^{2-2\theta _2,q})\overset{\eta /3}{\hookrightarrow }H^{1/3p,p}(0,T;H^{2-2\theta _2,q})\hookrightarrow L^{3p/2 }(0,T;H^{1,3q/2}(\Omega )).
\end{align*}
H\"older's inequality thus implies 
 \begin{align*}
\Vert v_1 \partial v_2\Vert _{L^p(L^q)}&\le \ \left \Vert \Vert v_1\Vert _{L^{3q}} \Vert \partial v_2\Vert _{L^{3q/2}} \right \Vert _{L^p}\le\Vert v_1\Vert _{L^{3p}(L^{3q})}\Vert v_2\Vert _{L^{3p/2}(H^{1,3q/2})}
\le CT^{\eta}\Vert v_1\Vert _{\IE _1}\Vert v_2\Vert _{\IE _1}.\qedhere
\end{align*}
\end{proof}

\vn
\begin{lemma}\label{le_space}
Let $q\in (1,\infty)$, $v_1,v_2\in H^{1+1/q,q}(\Omega)$ and  $w_1:=\int ^z_{-1}\divergence _Hv_1$. Then there exists a constant $C>0$ such that 
\[
\Vert w_1\partial _zv_2\Vert _{L^q}\le C\Vert v_1\Vert _{H^{1+1/q,q}}\Vert v_2\Vert _{H^{1+1/q,q}}.
\]
\end{lemma}

\begin{proof}
Similarly as in  \cite[Lemma 5.1]{Hieber2016} we obtain by anisotropic H\"{o}lder's inequality  and Sobolev inequalities 
\begin{align*}
\Vert w_1\partial _zv_2\Vert _{L^q}\le &\ \Vert w_1\Vert _{L^{2q}_{xy}L^\infty _z}\Vert \partial _zv_2\Vert _{L^{2q}_{xy}L^q _z}\\
\le &\ C\Vert \divergence _H v_1\Vert _{L^{2q}_{xy}L^1 _z}\Vert \partial _zv_2\Vert _{L^{2q}_{xy}L^q _z}\\
\le &\ C\Vert v_1\Vert _{H^{1+1/q}_{xy}L_z^1}\Vert v_2\Vert _{H^{1/q,q}_{xy}H^{1,q}_z}. \qedhere
\end{align*}
\end{proof}

\vn
\begin{lemma}\label{le_rhs_e0_diffi}
Let $p,q\in(1,\infty )$ such that $1/p+1/q\le 1$ and $\eta \in \left [0,1-\tfrac{1}{q}-\tfrac{1}{p}\right ]$. Then for all $v_1,v_2\in \IE _1(T)$ and $w_1$ given by  
$w_1:=\int ^z_{-1}\divergence _Hv_1$ there exists a constant $C>0$ such that 
\[ 
\Vert w_1 \partial _zv_2\Vert _{\IE _0(T)}\le CT^\eta\Vert v_1\Vert _{\IE _1(T)}\Vert v_2\Vert _{\IE _1(T)}.
\]
\end{lemma}

\begin{proof}
Setting  $\theta =\tfrac{\eta}{2}+\tfrac{1}{p}$, Lemma~\ref{le_emb} and the Mixed Derivative Theorem, Proposition~\ref{prop_mdt}, yield 
\begin{align*}
&\IE _1(T)\hookrightarrow H^{\theta ,p}(0,T;H^{2-2\theta ,q})\overset{\tfrac{\eta }{2}}{\hookrightarrow }H^{1/2p,p}(0,T;H^{2-2\theta ,q})\hookrightarrow L^{2p}(0,T;H^{1+1/q,q}(\Omega )),
\end{align*}
Putting $X:=H^{1+1/q,q}$ and $L^p(L^q):=\IE_0(T)$, Lemma~\ref{le_space} and the above embeddings imply
\begin{align*}
\Vert w_1 \partial _zv_2\Vert _{L^p(L^q)}&\le C\left \Vert \Vert v_1\Vert _{X} \Vert v_2\Vert _{X} \right \Vert _{L^p}\le  C\Vert v_1\Vert _{L^{2p}(X)}\Vert v_2\Vert _{L^{2p}(X)}\le CT^{\eta}\Vert 
v_1\Vert _{\IE _1(T)}\Vert v_2\Vert _{\IE _1(T)}.\qedhere
\end{align*}
\end{proof}

\subsection{Maximal regularity results}
In this subsection we prove that the vertical and horizontal solution of the primitive  equations belong to the maximal regularity class $\IE_1^u(T)$ as well as that 
the solution the  linearized system associated with  (\ref{eq_diff-eq}) fulfills a maximal regularity result estimate. We start by considering  the linearization of \eqref{eq_diff-eq}. It 
corresponds to the difference equation of (\ref{NSe}) and (\ref{eq_PE}).

Given $F\in \IE_0(T)$ and initial data $U_0\in X_\gamma$ we consider the linear problem 
\begin{equation}\label{eq_diff-eq_abstract}
\left \{\begin{array}{rll}
\partial _tU-\Delta U=& \ F-\nabla _\varepsilon P&\text{ in }(0,T)\times\Omega ,\\
\nabla _\varepsilon \cdot U=& \ 0&\text{ in }(0,T)\times\Omega ,\\
U ,P &\text{ periodic }&\text{ in }x,y,z,\\
U(0)=&\ U_0&\text{ in }\Omega ,
\end{array}\right .
\end{equation}
where $\nabla _\varepsilon :=(\partial _x,\partial _y,\varepsilon ^{-1}\partial _z)^T$. The functions $U$ and $P$ are the unknowns and represent the velocity and pressure differences.

We now aim to prove a maximal regularity estimate for $U$, where the constants are  \emph{independent} of the aspect ratio and the pressure gradient. 

\vn
\begin{lemma}\label{le_per_ex}
Let $\varepsilon >0$, $q\in (1,\infty )$ and assume that $F=(f_H,f_z)\in L^q(\Omega )$ and $P\in H^{1,q}_{per}(\Omega )$ are satisfying the equation
$-(\Delta _H +\varepsilon ^{-2}\partial _z^2)P=\divergence (f_H,\varepsilon ^{-1}f_z)$ for $\varepsilon>0$. Then there exists a constant $C>0$, independent of $\varepsilon$, such that
\[
\Vert (\nabla _HP,\varepsilon ^{-1}\partial _zP)\Vert _{q}\le C \Vert F\Vert _q.
\]
\end{lemma}

\begin{proof}
Denote by  $\cF$ the Fourier transform. For $k_\varepsilon  =(k _1,k _2,\varepsilon ^{-1}k _3)^T$ and $m(k) = -\frac{k  \otimes k }{\vert k \vert ^{2}}\in \IR^{3\times 3}$ set  
$m_\varepsilon (k)=m(k_\varepsilon)$. Then 
$$
(\nabla _H,\varepsilon ^{-1}\partial _z)P=\cF ^{-1}m_\varepsilon\cF F
$$ 
and $k \nabla m_\varepsilon (k)=k_\varepsilon  \nabla m(k_\varepsilon)$. Hence, 
\[
\sup _{\gamma \in  \{ 0, 1\}^3}\sup_{k \ne 0}\vert k ^\gamma D^\gamma m_\varepsilon(k)\vert =\sup _{\gamma \in  \{ 0, 1\}^3}\sup_{k _\varepsilon\ne 0}\vert k _\varepsilon^\gamma D^\gamma m(k_\varepsilon)\vert =1,
\]  
and Mikhlin's theorem in the period setting, see e.g. \cite[Proposition 4.5]{HECK20093739}, implies that $m_\varepsilon$ is an $L^p$-Fourier multiplier satisfying 
$\Vert \cF ^{-1}m_\varepsilon \cF \Vert _{\cL (L^q(\Omega ))} \leq C$ for some $C=C(q)>0$.
\end{proof}

\vn
\begin{proposition}\label{prop_mr_estimate}
Let $p,q\in (1,\infty )$, $T>0$, $F\in \IE _0(T)$, $U_0\in X_\gamma $ and $\varepsilon >0$. Then there is a unique solution $U,P$ to the equation
(\ref{eq_diff-eq_abstract}) with $U\in \IE_1(T)$ and $\nabla _\varepsilon P\in \IE_0(T)$, where $P$ is unique up to a constant. Moreover, there exist constants $C>0$ and $C_T>0$, independent of $\varepsilon$,  such that
\begin{equation*}
\Vert U\Vert _{\IE _1(T)}\le  C\Vert F\Vert _{\IE _0(T)}+C_T\Vert U_0\Vert _{X_\gamma }.
\end{equation*}
\end{proposition}

\begin{proof}
First, one defines the $\varepsilon$-dependent Helmholtz projection
\begin{align*}
\IP_{\varepsilon}:= \hbox{Id}- \nabla_{\varepsilon}\Delta_{\varepsilon}^{-1}\divergence_{\varepsilon}, \quad \hbox{where }  \Delta_{\varepsilon} = \nabla_{\varepsilon}\cdot \nabla_{\varepsilon},
\end{align*} 	
By Lemma~\ref{le_per_ex} this is a bounded projection with uniform norm bound independent of $\varepsilon$.

First, we we apply this to \eqref{eq_diff-eq_abstract}. Taking into account that due to periodicity $\IP_{\varepsilon}\Delta = \Delta\IP_{\varepsilon}$	and $\IP_{\varepsilon}\nabla_{\varepsilon} P=0$ hold, the equation \eqref{eq_diff-eq_abstract} reduces to the heat equation with right hand side $\IP_{\varepsilon}F$.

Maximal $L^p$-regularity of the three-dimensional Laplacian in the periodic setting  yields 
\[
\Vert U\Vert _{\IE _1(T)}\le   C\Vert \IP_{\varepsilon}F \Vert _{\IE _0(T)}+C_T\Vert U_0\Vert _{X_\gamma } \le  C\Vert F \Vert _{\IE _0(T)}+C_T\Vert U_0\Vert _{X_\gamma }.\qedhere
\]
%In order to prove the assertion it suffices thus to  show that $\Vert  \nabla _\varepsilon P\Vert _{\IE _0(T)}\le  C\Vert F\Vert _{\IE _0(T)}$. 
%This follows, however, %by a Mikhlin type argument on periodic domains as in  
%by Lemma~\ref{le_per_ex}. 
\end{proof}

Finally, we prove that the solution $u=(v,w)$ of the primitive equations belongs to the maximal regularity class $\IE_1^u(T)$. 

\vn
\begin{proposition}\label{prop_w_in_E1}
Let $p,q$ fulfill Assumption (A) and let $v$ be the strong solution of the primitive equations associated to  $v_0$ satisfying $(v_0,w_0)\in X_\gamma $. Then 
$$
u=(v,w)\in \IE _1^u(T) \mbox{ for all } T>0.
$$
\end{proposition}

\begin{proof}
It was shown in \cite[Theorem 3.3c]{GigaGriesHieberHusseinKashiwabara2017} that the primitive equations admit a unique solution $v\in\IE _1^v(T)$, which satisfies in addition 
$v\in C^\infty ((0,T),C^\infty (\overline{\Omega} )^2)$ and hence $w\in C^\infty ((0,T),C^\infty (\overline{\Omega} ))$ for any $T>0$. It remains to show that $w$ belongs to the maximal regularity class $\IE _1(T^*)$ for some $T^*>0$. 

Applying 
$\int _{-1}^z\divergence _H (\cdot)$ to (\ref{eq_PE}) yields  
\[
\partial _t w-\Delta w=f(v,w)\text{ in }(0 ,\infty )\times \Omega ,
\] 
where $f(v,w)=-\int _{-1}^z\divergence _H\left (\nabla _Hp+u\cdot\nabla  v\right )$. Using $\divergence _H\overline{v} =0$, for $z=1$ we obtain 
$2\Delta _H p=-\divergence _H\int_{-1}^1u\cdot\nabla v$ and thus 
\[
f(v,w)=\frac{1}{2}\int _{z}\divergence _Hu\cdot\nabla  v=\frac{1}{2}\int _{z}\divergence _H\divergence u\otimes v,
\]
where $\int _{z}:=\int_z^1-\int_{-1}^z+z\int_{-1}^1$. Observe that 
\[ 
\divergence _H\divergence u\otimes v=\partial _z (w\divergence _Hv+v\cdot \nabla _Hw)+(\divergence _Hv)^2+2v\cdot\nabla_H\divergence_Hv+\nabla_Hv\cdot(\nabla _Hv)^T.
\]
Hence, $f(v,w)=:f_1(v,w)+f_2(v)+f_3(v,w)$ with 
\[
f_1=(w\divergence _Hv-v\cdot \nabla_Hw)\vert_z^1,\quad f_2=\frac{1}{2}\int _{z}\left (\nabla _Hv\cdot (\nabla _Hv)^T+(\divergence _Hv)^2\right ),\quad f_3=\int _{z}\partial _zv\cdot \nabla_Hw.
\]
Here we used the fact that $\int_{z}v\cdot\nabla_H\divergence_Hv=-2(v\cdot\nabla_Hw)\vert _z^1+\int _z\partial _zv\cdot\nabla _Hw $, which follows by integration by parts. By Lemma~\ref{le_rhs_e0_easy} we obtain  $\Vert f_1\Vert _{\IE_0(T)}\le C\Vert v\Vert_{\IE_1(T)}\Vert w\Vert_{\IE_1(T)}$ and moreover 
\begin{align*}
\Vert f_2\Vert _{L^p(L^q)}\le &\ C\Vert \partial _zf_2\Vert _{L^p(L_{xy}^qL^1_z)}\le C\Vert v\Vert _{L^{2p}(H^{1,2q}_{xy}L_z^2)}^2,\\
\Vert f_3\Vert _{L^p(L^q)}\le &\ C\Vert \partial _zf_3\Vert _{L^p(L_{xy}^qL^1_z)}\le C\Vert v\Vert _{L^{2p}(L^{2q}_{xy}H_z^{1,2})}\Vert w\Vert _{L^{2p}(H^{1,2q}_{xy}L_z^2)}.
\end{align*}
For $1\ge 1/p+1/q$ we find by the Mixed Derivative Theorem and Sobolev's embeddings 
\begin{align*}
\IE_1\hookrightarrow H^{1/2p,p}(H^{2-1/p,q})\hookrightarrow H^{1/2p,p}(H^{1+1/q,q}_{xy}L^q_z\cap H^{1/q,q}_{xy}H^{1,q}_z)\hookrightarrow L^{2p}(H^{1,2q}_{xy}L^q_z\cap L^{2q}_{xy}H^{1,q}_z).
\end{align*}
If additionally $q\ge 2$, the above embedding implies $\IE_1\hookrightarrow L^{2p}(H^{1,2q}_{xy}L^2_z\cap L^{2q}_{xy}H^{1,2}_z).$
Similarly, for $q<2$ and $1\ge 4/3q+2/3p$ we find
\[
\IE_1\hookrightarrow H^{\frac{1}{2p},p}\!\left (H^{2-\frac{1}{p},q}\right )\hookrightarrow H^{\frac{1}{2p},p}\!\left (\! H^{1+\frac{1}{q},q}_{xy}H^{\frac{1}{q}-\frac{1}{2},q}_z\cap 
H^{\frac{1}{q},q}_{xy}H^{1+\frac{1}{q}-\frac{1}{2},q}_z\!\right)\hookrightarrow L^{2p}\!\left (H^{1,2q}_{xy}L^2_z\cap L^{2q}_{xy}H^{1,2}_z\right )\!.
\]
The above embeddings imply 
\begin{align*}
\Vert f_2(v)\Vert _{\IE_0(T)}\le C\Vert v\Vert ^2_{\IE_1(T)} \hbox{ and }\Vert f_1(v,w)\Vert _{\IE_0(T)}+\Vert f_3(v,w)\Vert _{\IE_0(T)}\le C\Vert v\Vert _{\IE_1(T)}\Vert w\Vert _{\IE_1(T)}.
\end{align*}
In particular $f_2(v)\in \IE_0(T)$. By maximal regularity there exists a solution operator $\cS :X_\gamma \times \IE_0(T)\rightarrow \IE_1(T)$ such that $u:=\cS (w_0,f_2(v))$ satisfies 
\[
\partial _tu-\Delta u=f_2\text{ in }(0,\infty )\times \Omega ,\quad u(0)=w_0.
\]
Setting now $B_v=-f_1(v,\cdot )-f_3(v,\cdot )\in \cL (\IE_1(T),\IE_0(T))$, i.e. it is a bounded linear operator from $\IE_1(T)$ to $\IE_0(T)$,  and adding $B_vu$ on both sides we see that 
\[ 
\partial _tu-\Delta u+B_vu=f_2+B_vu=[\II +B_v\cS (0,\cdot )]f_2+B_v\cS(w_0,0)\text{ in }(0,\infty )\times \Omega .
\]
Next, note that $\Vert \cS (0,\cdot )\Vert _{\cL (\IE_0(T),\IE_1(T))}$ can be bounded uniformly for $T\le 1$. Moreover $\Vert B_v\Vert _{\cL (\IE _1(T),\IE_0(T))}\le C\Vert v\Vert _{\IE_1(T)}$ by the previous 
estimates on $f_1$ and $f_3$. Choosing now $T^*$ small enough such that $\Vert v\Vert _{\IE_1(T^*)}<\left (C\Vert \cS (0,\cdot )\Vert _{\cL (\IE_0(T^*),\IE_1(T^*))}\right )^{-1},$ we see that 
$\Vert B_v\cS (0,\cdot )\Vert _{\cL (\IE_0(T^*))}<1$. A Neumann series argument yields $[\II+B_v\cS(0,\cdot )]^{-1}\in \cL(\IE_0(T^*))$, and thus 
$$
\tilde{u}:=\cS \big(w_0,[\II+B_v\cS(0,\cdot )]^{-1}(f_2-B_v\cS(w_0,0))\big)\in \IE_1(T^*)
$$ solves 
\begin{align*}
&\partial _t\tilde{u}-\Delta \tilde{u}+B_v\tilde{u}=\ [\II+B_v\cS(0,\cdot )]^{-1}(f_2-B_v\cS(w_0,0))+B_v\tilde{u}\\
=&\ [\II+B_v\cS(0,\cdot )]^{-1}(f_2-B_v\cS(w_0,0))+B_v\cS\big (0,[\II+B_v\cS(0,\cdot )]^{-1}(f_2-B_v\cS(w_0,0))\big )+B_v\cS (w_0,0)\\
=&\ [\II +B_v\cS (0,\cdot )][\II+B_v\cS(0,\cdot )]^{-1}(f_2-B_v\cS(w_0,0))+B_v\cS (w_0,0)=f_2
\end{align*} 
in $(0,T^* )\times \Omega $ with $\tilde{u}(0)=w_0$. Since  $f(v,\cdot)=f_2(v)-B_v$ and since the heat equation is uniquely solvable, we finally obtain 
$w=\tilde{u}\in \IE_1(T^*)$. Summing up, $w\in \IE _1(T)$ for any $T>0$.
\end{proof}

\vn
\begin{corollary}\label{coro_q_ineq}
Let $\mathcal{T}>0$ and $p,q\in (1,\infty )$ such that $1/p+1/q\le 1$. Let $(V_\varepsilon,W_\varepsilon )\in \IE^u_1(\mathcal{T})$ denote the solution of equation  \eqref{eq_diff-eq}
for some $u=(v,w)\in \IE_1(\mathcal{T})$ and initial data $U_0\in X_\gamma$. Then $X_\varepsilon (T)=\Vert (V_\varepsilon ,\varepsilon W_\varepsilon )\Vert _{\IE_1(T)}$ and for any 
$\eta \in \left [0,1-1/p-1/q\right ]$ there exists a constant $C>0$, independent of $\varepsilon$, such that  
\[ 
X_\varepsilon (T)\le CT^\eta \left [X_\varepsilon (T)\Vert u\Vert _{\IE_1(T)}+X_\varepsilon ^2(T)\right ]+\varepsilon C\left [\Vert u\Vert _{\IE_1(T)}+T^\eta \Vert u\Vert _{\IE_1(T)}^2\right ]+C\Vert U_0\Vert _{X_\gamma},
\]
for all $T\in [0,\mathcal{T}]$.
\end{corollary}

\begin{proof}
Since  
\[
F_H=-V_\varepsilon \cdot\nabla _Hv-W_\varepsilon\partial _zv-v\cdot\nabla _HV_\varepsilon-w\partial _zV_\varepsilon-V_\varepsilon\cdot\nabla _HV_\varepsilon-W_\varepsilon\partial _zV_\varepsilon,
\]
we obtain with the help of Lemma~\ref{le_rhs_e0_easy} and \ref{le_rhs_e0_diffi}
\begin{equation}\label{eq_rhs_e0_fh}
\Vert F_H\Vert _{\IE _0(T)}\le CT^\eta \Vert V_\varepsilon\Vert _{\IE _1(T)}(\Vert V_\varepsilon\Vert _{\IE _1(T)}+\Vert (v,w)\Vert _{\IE _1(T)}).
\end{equation}
Similarly, since
\[
\varepsilon F_z=\varepsilon (-V_\varepsilon \cdot\nabla _Hw-w\divergence _HV_\varepsilon )-\varepsilon W_\varepsilon\divergence _H (v+V_\varepsilon )-(v+V_\varepsilon )\cdot
\nabla _H\varepsilon W_\varepsilon-\varepsilon (\partial _tw-u\cdot\nabla w+\Delta w),
\]
Lemma~\ref{le_rhs_e0_easy} yields 
\begin{align*}
 \Vert \varepsilon F_z\Vert _{\IE _0(T)}
\le  CT^\eta [\Vert V_\varepsilon \Vert _{\IE _1(T)}\Vert w \Vert _{\IE _1(T)}+\Vert \varepsilon W_\varepsilon \Vert _{\IE _1(T)}(\Vert V_\varepsilon 
\Vert _{\IE _1(T)}+\Vert v \Vert _{\IE _1(T)})]+C\varepsilon [\Vert w \Vert _{\IE _1(T)}+T^\eta \Vert w \Vert _{\IE _1(T)}^2].
\end{align*}
Combining this estimate with \eqref{eq_rhs_e0_fh}, Proposition~\ref{prop_mr_estimate} yields the assertion. %
\end{proof}

\subsection{Proof of the main result}

\begin{proof}[{Proof of Theorem~\ref{mthm_MR}}]
Fix $\mathcal{T}>0$ and denote by  $u$ the solution of equation (\ref{eq_PE}). Proposition~\ref{prop_w_in_E1} yields 
$u\in \IE_1(\mathcal{T})$. We now show that 
$$
X_\varepsilon (T):=\Vert V_\varepsilon ,\varepsilon W_\varepsilon \Vert _{\IE_1(T)}\le \varepsilon C(\Vert (v,w) 
\Vert _{\IE _1(\mathcal{\cdot})},\mathcal{T},p,q)
$$ 
for all $T\in [0,\mathcal{T}]$ and $\varepsilon >0$ small enough. By the uniform continuity of 
$T\mapsto \Vert u\Vert _{\IE_1 (T)}$ on $[0,\mathcal{T}]$ there is a $T^*\in [0,\mathcal{T}]$ such that 
$\Vert u\Vert _{\IE_1 (T+T^*)}^p-\Vert u\Vert _{\IE_1 (T)}^p\le (2C\mathcal{T}^\eta )^{-p}$ for all $T\in [0,\mathcal{T}-T^*]$ and  where  $C$ denotes the constant given in 
Corollary~\ref{coro_q_ineq}. The latter  with $U_0=0$ implies 
\begin{align}\label{eq:qudratic}
CX^2_\varepsilon(T)-\tfrac{1}{2} X_\varepsilon (T)+\varepsilon \ge 0,\quad T\in [0,T^*].
\end{align}
Observe  that $(X_\varepsilon )^p:t\mapsto\left (\int_0^t\ldots\right )$ is continuous in $[0,T]$ and $X_\varepsilon (0)=0$. Thus, for $\varepsilon <(16C)^{-1}$, we may solve this quadratic inequality 
and  obtain $X_\varepsilon \le 2\varepsilon $ on $[0,T^*]$.

Note that inequality~\eqref{eq:qudratic} holds indeed on a time interval independent of $\varepsilon$. More specifically, if one replaces $T^*$ by  $T_{\varepsilon}<T^*$, where $T_{\varepsilon}$ is the maximal existence time of \eqref{NSe} with initial data $u_{0}$, then it holds similarly as above that   
$X_\varepsilon \le 2\varepsilon $ on $[0,{T}_{\varepsilon}]$ which implies a contradiction to the maximality of the existence time.   

Assume there is some $m\in \IN $ such that $mT^*<\mathcal{T}$ and $X_\varepsilon \le \varepsilon 2K_m$ in $[0,mT^*]$,  where $K_1=1$ and 
$K_m=2^{1/p}\left [\left (2Cc_{mT^*}+1\right )K_{m-1}+1\right ]$ and  $c_{T}$ denotes the embedding constant of $\IE _1(T)\hookrightarrow L^\infty (0,T;X_\gamma )$. Let 
$(\tilde{V}_\varepsilon ,\varepsilon\tilde{W}_\varepsilon )(T)=(V_\varepsilon ,\varepsilon W_\varepsilon )(T+mT^*)$ be the unique solution of problem (\ref{eq_diff-eq}) with respect to 
$\tilde{u}(T)=u(T+mT^*)$ and  initial data $U_0=({V}_\varepsilon ,\varepsilon {W}_\varepsilon )(mT^*)$. Setting 
$$\tilde{X}_\varepsilon ^p (T):=\Vert (\tilde{V}_\varepsilon ,\varepsilon \tilde{W}_\varepsilon )\Vert _{\IE _1(T)}^p={X}_\varepsilon ^p(T+mT^*)-{X}_\varepsilon ^p(mT^*),$$ 
Corollary~\ref{coro_q_ineq} and the argument about the $\varepsilon$-independency of the time interval given above  imply 
\[C
\tilde{X}_\varepsilon ^2(T)-\tfrac{1}{2} \tilde{X}_\varepsilon (T)+\varepsilon +C\Vert U_0\Vert _{X_\gamma }\ge 0,\quad T\in [0,\min \{T^*;\mathcal{T}-mT^*\}].
\]
By assumption $\Vert U_0\Vert _{X_\gamma }\le c_{mT^*}{X}_\varepsilon (mT^*)\le c_{mT^*}\varepsilon 2K_m$. Since $\tilde{X}_\varepsilon (0)=0$ and $\tilde{X}_\varepsilon$ is continuous 
in $[0,\min \{T^*;\mathcal{T}-mT^*\}]$, we may solve the quadratic inequality for $\varepsilon <(16C(1+2Cc_{mT^*}K_m))^{-1}$ and obtain 
$\tilde{X}_\varepsilon\le \varepsilon 2(1+2Cc_{mT^*}K_m)$ in $[0,\min \{T^*;\mathcal{T}-mT^*\}]$. Hence, by the assumption on $m$ 
\[
X_\varepsilon ^p(T)\le \left (\varepsilon 2(1+2Cc_{mT^*}K_m)\right )^p+X_\varepsilon ^p(mT^*)\le 2\left [2\varepsilon (1+2Cc_{mT^*}K_m)+\varepsilon 2K_m\right ]^p=(\varepsilon 2K_{m+1})^p,
\]
for all $T\in [mT^*,\min \{(m+1)T^*;\mathcal{T}\}]$. The assumption on  $m$ implies  $X_\varepsilon \le \varepsilon 2K_{m+1}$ in $[0,\min \{(m+1)T^*;\mathcal{T}\}]$. By induction we 
get $X_\varepsilon \le \varepsilon 2K_{M}$ in $[0,\mathcal{T}]$ with $M=\left \lceil \frac{\mathcal{T}}{T^*}\right \rceil $. The proof of Theorem~\ref{mthm_MR} is complete.
\end{proof}

\vn
\begin{remarks} \label{remunif}
a) It is remarkable that in  every induction step we do \emph{not} rely upon the local well-posedness of equation (\ref{eq_diff-eq}). In fact,  the boundedness of the difference and the 
long time well-posedness of the primitive equations as well as the local well-posedness of the Navier-Stokes equations are sufficient for our arguments.  More specifically, let 
$X_\varepsilon \le \varepsilon 2K_{m}$ on $[0,mT^*]$. To construct a solution $(\tilde{V}_\varepsilon ,\varepsilon \tilde{W}_\varepsilon)$ to (\ref{eq_diff-eq}) with inhomogeneous initial data 
$(V_\varepsilon ,\varepsilon W_\varepsilon )(mT^*)$ we reconstruct the solution to the Navier-Stokes equations with initial data $(V_\varepsilon ,\varepsilon W_\varepsilon )(mT^*)+u(mT^*)$, which exists 
locally, from the solution of the primitive equations with initial data $u(mT^*)$. This method ensures local existence of the difference solution 
$(\tilde{V}_\varepsilon ,\varepsilon \tilde{W}_\varepsilon)$. 

On the other hand every solution to (\ref{eq_diff-eq}) with initial data $(V_\varepsilon ,\varepsilon W_\varepsilon )(mT^*)$ adds up with the solution of the primitive equations with initial data $u(mT^*)$ 
to a solution of the Navier-Stokes equation with initial data $(V_\varepsilon ,\varepsilon W_\varepsilon )(mT^*)+u(mT^*)$.

b) Given $u\in \IE _1(\infty )$, we may adjust the above proof in such a way that $X_\varepsilon (T) \le \varepsilon C$ \emph{uniformly} for all $T\in (0,\infty ]$. More precisely, 
there exist \emph{finitely} many $0=T_0< T_1< \ldots < T_m=\infty $ such that $\Vert u\Vert _{\IE_1 (T_i)}^p-\Vert u\Vert _{\IE_1 (T_{i-1})}^p\le (2C)^{-p}$ for $i=1,\ldots m$. Proceeding as above, while 
using Corollary~\ref{coro_q_ineq}, $\eta =0$ yields $X_\varepsilon (T) \le \varepsilon C$ for all $T\in [0,\infty )$, where $C$ is \emph{independent} of $T$. Taking the limit yields the assertion 
for $T=\infty$.
\end{remarks}

%%----------------------INSERT FILE--------------------------
%\nocite{*}
%\bibliographystyle{alphadin}
%%\bibliographystyle{unsrt}
%\bibliography{biblio}

%%----------------------INSERT .tex-CODE---------------------

\end{document}